\documentclass{conm-p-l}

\usepackage{amssymb}

\DeclareMathOperator{\wt}{wt}
\newcommand{\nc}{\newcommand}

\nc{\op}{\oplus} \nc{\pv}{P^{\vee}}

\newcommand{\C}{\mathbb{C}}

\newcommand{\Z}{\mathbb{Z}}
\newcommand{\ZZ}{\mathbb{Z}_{\geq 0}}

\newcommand{\la}{\lambda}

\newcommand{\ep}{\varepsilon}

\newcommand{\BB}{B}

\nc{\B}{\mathbf{B}} \nc{\V}{\mathbf{V}}
\nc{\nbinom}[2]{\genfrac{}{}{0pt}{1}{{#1}}{{#2}}}
\nc{\qbinom}[2]{\left[\genfrac{}{}{0pt}{1}{{#1}}{{#2}}\right]}
\nc{\ft}{\tilde{f}} 
\nc{\et}{\tilde{e}} 
\nc{\Y}{\mathbf{Y}}
\nc{\ra}{\rightarrow} 
\nc{\vep}{\varepsilon} 
\nc{\vp}{\varphi}
\nc{\h}{\mathfrak{h}} 
\nc{\oP}{\overline{P}}
\nc{\Fit}{\tilde{F}_i} 
\nc{\Eit}{\tilde{E}_i}
\nc{\fit}{\tilde{f}_i} 
\nc{\eit}{\tilde{e}_i}

\newtheorem{theorem}{Theorem}[section]
\newtheorem{lemma}[theorem]{Lemma}
\newtheorem{proposition}[theorem]{Proposition}

\theoremstyle{definition}
\newtheorem{definition}[theorem]{Definition}

\theoremstyle{remark}

\numberwithin{equation}{section}

\begin{document}

\title{A note on  $U_q(D_4^{(3)})$ - Demazure crystals}


\author[A. Armstrong]{Alyssa M. Armstrong}
\address{Department of Mathematics, North Carolina State University,  Raleigh,  
NC 27695-8205}
\email{amarmstr@ncsu.edu}

\author[K.C. Misra]{ Kailash C. Misra}
\address{Department of Mathematics, North Carolina State University,  Raleigh,  
NC 27695-8205}
\email{misra@ncsu.edu}
\thanks{Partially supported by NSA grant, H98230-12-1-0248.}

\subjclass[2010]{Primary 17B37,17B10; Secondary 17B67}




\begin{abstract}
We show that there exists a suitable sequence $\{w^{(k)}\}_{k \ge 0}$ of Weyl group elements for the perfect crystal $B = B^{1,3l}$ such that the path realizations of the Demazure crystals $B_{w^{(k)}}(l\Lambda_2)$ for the quantum affine algebra $U_q(D_4^{(3)})$ have tensor product like structure with mixing index $\kappa =1$.
\end{abstract}

\maketitle

\section{Introduction}
\setcounter{equation}{0}

Consider the quantum affine algebra  $U_q(\mathfrak g)$ (cf. \cite{Lu2}) associated with the affine Lie algebra $\mathfrak g$ (cf. \cite{Kac}).
For a dominant integral weight $\lambda$ of level $l > 0$, let $V(\lambda)$ denote the integrable highest weight $U_q(\mathfrak g)$-module with highest weight $\lambda$ and the pair $(L(\lambda), B(\lambda))$ denote its crystal base (\cite{Ka1}, \cite{Ka2}, \cite{Lu1}). The crystal $B(\lambda)$ has many interesting combinatorial properties and can be realized as elements (called {\it paths}) in the semi-infinite tensor product $\cdots \otimes B \otimes B \otimes B$ where $B$ is a perfect crystal of level $l$ for $U_q(\mathfrak g)$ \cite{KMN1}. A perfect crystal $B$ of level $l$ for the quantum affine algebra $U_q(\mathfrak g)$ can be thought of as a crystal for a level zero representation of the derived subalgebra $U'_q(\mathfrak g)$ with certain properties. 

Let $W$ denote the Weyl group for $\mathfrak g$ and $U^+_q(\mathfrak g)$ denote the positive part of $U_q(\mathfrak g)$. For each $w \in W$, the $U^+_q(\mathfrak g)$-submodule $V_w(\lambda)$ generated by the one-dimensional extremal weight space $V(\lambda)_{w\lambda}$ of $V(\lambda)$ is called a Demazure module. In \cite{Ka3}, Kashiwara showed that the Demazure crystal $B_w(\lambda)$ for the Demazure module $V_w(\lambda)$ is a subset of the crystal $B(\lambda)$ for $V(\lambda)$ satisfying a certain recursive property. In \cite{KMOU}, it has been shown that under certain conditions,  the Demazure crystal $B_w(\Lambda)$ as a subset of $\cdots \otimes B \otimes B \otimes B$ has tensor product-like structure. In this realization a certain parameter $\kappa$, called the {\it mixing index} enters the picture. It is conjectured that for all quantum affine algebras $\kappa \leq 2$. It is known that for $\lambda = l\Lambda$, where $\Lambda$ is a dominant weight of level $1$ the mixing index $\kappa = 1$ for ${\mathfrak g} = A_n^{(1)}, B_n^{(1)}, C_n^{(1)}, D_n^{(1)}, A_{2n-1}^{(2)}, D_{n+1}^{(2)}, A_{2n}^{(2)}$ \cite{KMOTU}, $D_4^{(3)}$ \cite{M} and $G_2^{(1)}$ \cite{JM}.

In this paper using the perfect crystal $B= B^{1,3l}$ of level $3l$ constructed in \cite{KMOY} for ${\mathfrak g} = D_4^{(3)}$ we show that  there exists a suitable sequence $\{w^{(k)}\}_{k \ge 0}$ of Weyl group elements in $W$ such that the Demazure crystals $B_{w^{(k)}}(l\Lambda_2)$ has tensor product like structures with $\kappa = 1$. We observe that the dominant weight $\Lambda_2$ has level $3$.  


\section{Quantum affine algebras and the perfect crystals}
\label{affine crystals}
\setcounter{equation}{0}

In this section we recall necessary facts in crystal base theory for quantum affine algebras. Our basic references for this section are \cite{Kac}, \cite{Lu2}, \cite{Ka2}, \cite{KMN1}, and \cite{KKM}.

Let $I = \{0,1,...,n\}$ be the index set and let $A=(a_{ij})_{i,j
\in I}$ be an affine Cartan matrix and $D={\rm diag}(s_0, s_1, \ldots,s_n)$ be a diagonal matrix with all $s_i \in \mathbb Z_{>0}$ such that $DA$ is symmetric. The {\it dual weight lattice} $P^{\vee}$ is defined to be the free abelian group
$P^{\vee} = \Z h_0 \oplus \Z h_1 \oplus \cdots \oplus \Z h_n \oplus
\Z d$ of rank $n+2$, whose complexification $\mathfrak{h}=\C \otimes
P^{\vee}$ is called the {\it Cartan subalgebra}. We define the
linear functionals $\alpha_i$ and $\Lambda_i$ $(i \in I)$ on
$\mathfrak{h}$ by
$$\alpha_i(h_j)=a_{ji}, \ \ \alpha_i(d)=\delta_{i0}, \ \
\Lambda_i (h_j)=\delta_{ij}, \ \   \Lambda_i(d)=0 \ \  (i,j \in I ).$$ 
The $\alpha_i$'s are called the {\it simple roots} and the $\Lambda_i$'s are
called the {\it fundamental weights}. We denote by
$\Pi=\{\alpha_i~|~i\in I\}$ the set of simple roots. We also define
the {\it affine weight lattice} to be $P=\{\lambda \in
\mathfrak{h}^* ~|~ \lambda(P^{\vee})\subset \Z\}$. The quadruple
$(A, P^{\vee}, \Pi, P)$ is called an {\it affine Cartan datum}. We
denote by $\mathfrak{g}$ the affine Kac-Moody algebra corresponding
to the affine Cartan datum $(A, P^{\vee}, \Pi, P)$ (see
\cite{Kac}). Let $\delta$ denote the {\it null root} and $c$
denote the {\it canonical central element} for $\mathfrak{g}$ (see
\cite[Ch. 4]{Kac}).  Now the affine weight lattice can be written as
$P = \mathbb{Z} \Lambda_0 \oplus \mathbb{Z} \Lambda_1 \oplus \cdots
\oplus \mathbb{Z}\Lambda_n \oplus \mathbb{Z}\delta $. Let $P^{+} =
\{ \lambda \in P ~|~ \lambda(h_i) \geq 0$ for all $i \in I\}$. The
elements of $P$ are called the {\it affine weights} and the elements
of $P^{+}$ are called the {\it affine dominant integral weights}.

Let $\bar{P}^{\vee} = \mathbb{Z}h_0 \oplus \cdots \oplus
\mathbb{Z}h_n$, $\bar{\mathfrak{h}} = \mathbb{C}
\otimes_{\mathbb{Z}} \bar{P}^{\vee}$, $\bar P = \mathbb{Z} \Lambda_0
\oplus \mathbb{Z} \Lambda_1 \oplus \cdots \oplus
\mathbb{Z}\Lambda_n$ and $\bar{P}^{+} = \{\lambda \in \bar{P} ~|~
\lambda(h_i) \geq 0$ for all $i \in I\}$. The elements of $\bar{P}$
are called the {\it classical weights} and the elements of
$\bar{P}^{+}$ are called the {\it classical dominant integral
weights}. The {\it level} of a (classical) dominant integral weight
$\lambda$ is defined to be $l = \lambda(c)$.  We call the
quadruple $(A, \bar{P}^{\vee}, \Pi, \bar{P})$ the {\it classical
Cartan datum}.

For the convenience of notation, we define
$\left[k\right]_x=\dfrac{x^k-x^{-k}}{x-x^{-1}}$, where $k$ is an
integer and $x$ is a symbol. We also define $\left[\begin{matrix} m
\\ k \end{matrix}
\right]_x=\dfrac{\left[m\right]_x!}{\left[k\right]_x!\left[m-k\right]_x!}$,
where $m$ and $k$ are nonnegative integers, $m \geq k \geq 0$,
$\left[k\right]_x!=\left[k\right]_x\left[k-1\right]_x \cdots
\left[1\right]_x$ and $\left[0\right]_x!=1$.

The {\it quantum affine algebra} $U_q(\mathfrak{g})$ is the quantum
group associated with the affine Cartan datum $(A, P^{\vee}, \Pi,
P)$. That is, it is the associative algebra over $\C(q)$ with unity
generated by $e_i, f_i(i\in I)$ and $q^h( h \in P^{\vee})$
satisfying the following defining relations:
\begin{itemize}

\item[(i)] $q^0 = 1, q^h q^{h^{\prime}} = q^{h + h^{\prime}}$  for all  $h,h^{\prime} \in P^{\vee}$,

\item[(ii)] $q^h e_i q^{-h} = q^{\alpha_i(h)}e_i$,  $q^h f_i q^{-h} = q^{-\alpha_i(h)}f_i$  for  $h \in P^{\vee}$,

\item[(iii)] $e_i f_j - f_j e_i = \delta_{ij} \dfrac{K_i - K_i^{-1}}{q_i - q_i^{-1}}$  for  $i,j \in I$,  where  $q_i = q^{s_i}$  and  $K_i = q^{s_i h_i}$,

\item[(iv)]$\displaystyle\sum_{k=0}^{1-a_{ij}}(-1)^k e_{i}^{(1-a_{ij}-k)} e_j e_{i}^{(k)} = 0 $ for $i \ne j$,

\item[(v)]$\displaystyle\sum_{k=0}^{1-a_{ij}}(-1)^k f_{i}^{(1-a_{ij}-k)} f_j f_{i}^{(k)} = 0 $  for $i \ne j$,
\end{itemize}

where $e_i^{(k)}= \frac{e_i^k}{\left[k\right]_{q_i}!}$,
and $f_i^{(k)} = \frac{f_i^k}{\left[ k \right]_{q_i}!}$.
We denote by $U_q'(\mathfrak{g})$ the subalgebra of
$U_q(\mathfrak{g})$ generated by $e_i,f_i,K_i^{\pm1}(i \in I)$. The algebra $U_q'(\mathfrak{g})$ can be regarded as
the quantum group associated with the classical Cartan datum $(A, \bar{P}^{\vee}, \Pi, \bar{P})$.

\begin{definition}
An {\it affine crystal} (respectively, a {\it classical crystal}) is
a set $\BB$ together with the maps $\wt : \BB \rightarrow P$
(respectively, $\wt: \BB \rightarrow \bar{P}$), $\tilde{e_i},
\tilde{f_i}: \BB \rightarrow \BB \cup \{ 0 \}$ and $\varepsilon_i,
\varphi_i : \BB \rightarrow \mathbb{Z} \cup \{-\infty\}$  $(i \in I)
$ satisfying the following conditions:
\begin{itemize}
\item[(i)] $\varphi_{i}(b) = \varepsilon_{i}(b) + \langle h_i, \wt
(b) \rangle$ \ for all  $ i \in I$,

\item[(ii)] $ \wt(\tilde {e_i} b) = \wt (b) + \alpha_i$ \ if  $\tilde{e_i} b \in \BB$,

\item[(iii)] $ \wt(\tilde{f_i} b) = \wt(b) - \alpha_i$ \ if  $\tilde{f_i}b \in \BB$,

\item[(iv)] $\varepsilon_i(\tilde{e_i} b) = \varepsilon_i(b) - 1$, \
$\varphi_i(\tilde{e_i} b) = \varphi_i(b) + 1$ \ if  $\tilde{e_i}b \in
\BB$,

\item[(v)] $\varepsilon_i(\tilde{f_i} b) = \varepsilon_i(b) + 1$, \
$\varphi_i(\tilde{f_i} b) = \varphi_i(b) - 1$ \ if  $\tilde{f_i}b \in
\BB$,

\item[(vi)] $\tilde{f_{i}} b = b^{'}$ \ \ if and only if \ \ $b =
\tilde{e_i} b^{'}$ for  $b,b^{'}\in \BB, i \in I $,

\item[(vii)] If $\varphi_i(b) = -\infty$  for  $b\in \BB$ , then
$\tilde{e_i} b = \tilde{f_i} b = 0$.

\end{itemize}
\end{definition}

\begin{definition}
Let $\BB_1$ and $\BB_2$ be affine or classical crystals.
A {\it crystal morphism} (or {\it
morphism of crystals}) $\Psi : \BB_1 \rightarrow \BB_2$ is a map
$\Psi : \BB_1 \cup \{0\} \rightarrow \BB_2 \cup \{0\}$ such that
\begin{enumerate}
\item[(i)] $\Psi(0)=0$,
\item[(ii)] if $b \in \BB_1$ and $\Psi(b) \in \BB_2$, then $\wt(\Psi(b))=\wt(b)$, $\varepsilon_i(\Psi(b))=\varepsilon_i(b)$, and $\varphi_i(\Psi(b))=\varphi_i(b)$ for all $i \in I$,
\item[(iii)] if $b, b' \in \BB_1$, $\Psi(b), \Psi(b') \in \BB_2$ and $\fit b=b'$, then $\fit \Psi(b)=\Psi(b')$ and $\Psi(b)=\eit \Psi(b')$ for all $i \in I$.
\end{enumerate}
A crystal morphism $\Psi : \BB_1 \rightarrow \BB_2$ is called an {\it isomorphism} if it is a bijection from $\BB_1 \cup \{0\}$ to $\BB_2 \cup \{0\}$.
\end{definition}

For crystals $\BB_1$ and $\BB_2 $, we define the {\it tensor
product} $\BB_1 \otimes \BB_2$ to be the set $\BB_1 \times \BB_2 $
whose crystal structure is given as follows:

\begin{equation*}
\begin{aligned}
\tilde {e_i}(b_1 \otimes b_2)&= \begin{cases} \tilde {e_i} b_1 \otimes b_2 &\text{if \;$\varphi_i(b_1)\ge \varepsilon_i(b_2)$}, \\
b_1\otimes \tilde {e_i} b_2 & \text{if\; $\varphi_i(b_1) < \varepsilon_i(b_2)$}, \end{cases} \\
\tilde {f_i}(b_1 \otimes b_2)&= \begin{cases} \tilde {f_i} b_1 \otimes b_2
&\text{if \; $\varphi_i(b_1) > \varepsilon_i(b_2)$}, \\
b_1\otimes \tilde {f_i} b_2 & \text{if \; $\varphi_i(b_1) \le
\varepsilon_i(b_2)$},
\end{cases}\\
\wt(b_1 \otimes b_2)&= \wt(b_1)+\wt(b_2),\\
\ep_{i}(b_1 \otimes b_2)&= \max(\ep_{i}(b_1), \ep_i(b_2) -
\langle h_i, \wt(b_1) \rangle ),\\
\varphi_{i}(b_1 \otimes b_2)& = \max(\varphi_{i}(b_2), \varphi_i(b_1)
+\langle h_i, \wt(b_2)\rangle ).
\end{aligned}
\end{equation*}\\

Let $\BB$ be a classical crystal. For an element $b \in \BB$, we
define
\begin{eqnarray*}
\varepsilon(b) = \displaystyle\sum_{i \in I} \varepsilon_i(b)\Lambda_i, &
\varphi(b) = \displaystyle\sum_{i \in I} \varphi_i(b)\Lambda_i.
\end{eqnarray*}

\begin{definition} \label{defi:perfect crystal}
Let $l$ be a positive integer. A classical crystal $\BB$ is
called a \textit{perfect crystal of level $l$} if

\begin{itemize}
\item[(1)] there exists a finite dimensional $U'_q (\mathfrak{g})$-module with a crystal basis whose crystal graph is isomorphic to $\BB$,

\item[(2)] $\BB \otimes \BB$ is connected,

\item[(3)] there exists a classical weight $ \lambda_0 \in \bar{P}$
such that $\displaystyle \wt(\BB) \subset \lambda_0 + \sum_{i \ne 0}
\mathbb{Z}_{\le0} \: \alpha_i, \;\\
\#(\BB_{\lambda_0})=1 $, where
$\BB_{\lambda_0}=\{ b \in \BB ~ | ~ \wt(b)=\lambda_0 \}$,

\item[(4)] for any $b \in \BB , ~  \langle c, \varepsilon(b) \rangle
\ge l $,

\item[(5)] for any $ \lambda \in \bar{P} ^{+}$ with $\la(c)= l$, there
exist  unique $ b^{\lambda} , b_{\lambda} \in \BB $ such that
$\varepsilon(b^{\lambda}) = \lambda = \varphi(b_{\lambda}).$

\end{itemize}
\end{definition}

The following crystal isomorphism theorem plays a fundamental role
in the theory of perfect crystals.

\begin{theorem}\cite{KMN1}
Let $\BB$ be a perfect crystal of level $l$ ($l \in
\mathbb{Z}_{\geq 0}$). For any $\lambda \in \bar{P}^+$ with
$\lambda(c)= l$, there exists a unique classical crystal
isomorphism
\begin{equation*}
\begin{aligned}
\Psi :\BB(\lambda) \stackrel{\sim}{\longrightarrow}
\BB(\ep(b_\lambda)) \otimes \BB &
&\text{given by}& &  u_\lambda  \longmapsto  u_{\ep(b_\lambda)} \otimes b_\lambda ,
\end{aligned}
\end{equation*}
where $u_\lambda$ is the highest weight vector in $\BB(\lambda)$ and $b_\lambda$ is the unique vector in $\BB$ such that $\varphi(b_{\la})=\la $.
\end{theorem}
 Set $\la_0 =\la, \ \la_{k+1}=\ep(b_{\la_k}), \ b_0=b_{\la_0}, \ b_{k+1}=b_{\la_{k+1}}$.
 Applying the above crystal isomorphism repeatedly, we get a sequence of crystal isomorphisms
 \begin{equation*}
\begin{array}{ccccccc} \BB(\lambda) &
\stackrel{\sim}{\longrightarrow} & \BB(\lambda_1) \otimes \BB &
\stackrel{\sim}{\longrightarrow} & \BB(\lambda_2)\otimes \BB \otimes
\BB & \stackrel{\sim}{\longrightarrow} & \cdots
\\
u_\lambda & \longmapsto & u_{\lambda_1} \otimes b_0 & \longmapsto &
u_{\lambda_2} \otimes b_1 \otimes b_0 & \longmapsto & \cdots.
\end{array} \end{equation*}
 In this process, we get an infinite sequence $\mathbf{p}_\lambda =(b_k)^{\infty} _{k=0} \in \BB^{\otimes \infty}
 $, which is called the \textit{ground-state path of weight} $\la$.
 Let $\mathcal{P}(\la):=\{\mathbf{p}=(p(k))^{\infty}_{k=0} \in \BB^{\otimes \infty} \;|\; p(k) \in \BB, p(k)=b_k \ \text{for all} \ k \gg 0  \}$.
 The elements of $\mathcal{P}(\la)$ are called the $\la$-\textit{paths}. The following result gives the {\it path realization} of $\BB(\lambda)$.

\begin{proposition}\cite{KMN1}
There exists an isomorphism of classical crystals
\begin{equation*}
\begin{aligned}
\Psi_{\la} :\BB(\lambda) & \stackrel{\sim}{\longrightarrow} &
\mathcal{P}({\la}) &
&\text{given by}& &  u_\lambda & \longmapsto & \mathbf{p}_\lambda  ,
\end{aligned}
\end{equation*}
where $u_{\la}$ is the highest weight vector in $\mathcal B(\la)$.
\label{path_realization}
\end{proposition}


\section{$U_q(D_4^{(3)})$- Perfect crystals}
\setcounter{equation}{0}

In this section we recall the perfect crystal $B^{1,l}$ for the quantum affine algebra
$U_q(D_4^{(3)})$ of level $l>0$ constructed in \cite{KMOY}.

First we fix the data for $D_4^{(3)}$.
Let  $\{\alpha_0, \alpha_1, \alpha_2\}$, 
$\{h_0, h_1, h_2\}$ and 
$\{\Lambda_0, \Lambda_1, \Lambda_2\}$ be the set of 
simple roots, simple coroots and fundamental weights, respectively.
The Cartan matrix $A=(a_{ij})_{i,j=0,1,2}$
is given by
\[A=
\left(
\begin{array}{rrr}
2  & -1 & 0  \\
-1 & 2  & -3 \\
0  & -1 & 2
\end{array}
\right),
\]

and its Dynkin diagram is given by:

\[
{} \circ_0-{}\circ_1 \Lleftarrow \circ_2
\]

The standard null root $\delta$ 
and the canonical central element $c$ are 
given by
\[
\delta=\alpha_0+2\alpha_1+\alpha_2
\quad\text{and}\quad c=h_0+2h_1+3h_2, 
\]
where 
$\alpha_0=2\Lambda_0-\Lambda_1+\delta,\quad
\alpha_1=-\Lambda_0+2\Lambda_1-\Lambda_2,\quad
\alpha_2=-3\Lambda_1+2\Lambda_2.$

For positive integer $l$ define the set

\begin{eqnarray*}
B=B^{1,l}=\left\{
b=(x_1,x_2,x_3,{\bar x}_3,{\bar x}_2,{\bar x}_1)
\in(\ZZ)^6
\left\vert
\begin{array}{l}
x_3\equiv\bar{x}_3\;(\text{mod }2), \\
\sum_{i=1,2} (x_i+{\bar x}_i)+\frac{x_3+{\bar x}_3}{2}
\leq l
\end{array}
\right.
\right\}.
\end{eqnarray*}

For $b=(x_1,x_2,x_3,{\bar x}_3,{\bar x}_2,{\bar x}_1) \in B$ we denote
\begin{equation} \label{def s(b)}
s(b)=x_1+x_2+\frac{x_3+{\bar x}_3}{2}+{\bar x}_2+{\bar x}_1, t(b) = x_2+\frac{x_3+{\bar x}_3}{2} ,
\end{equation}

\begin{equation} \label{z1-4}
z_1={\bar x}_1-x_1, \quad 
z_2={\bar x}_2 -{\bar x}_3, \quad 
z_3=x_3-x_2, \quad 
z_4= \frac{{\bar x}_3-x_3}{2},
\end{equation}

and 
\begin{equation} \label{A}
\mathcal{A}=(0,z_1,z_1+z_2,z_1+z_2+3z_4,z_1+z_2+z_3+3z_4,2z_1+z_2+z_3+3z_4).
\end{equation}

Now we define conditions ($E_1$)-($E_6$) and ($F_1$)-($F_6$) as follows.

\begin{equation} \label{(F)}
\begin{cases}
&(F_1)\quad
z_1+z_2+z_3+3z_4\le0, z_1+z_2+3z_4\le0, z_1+z_2\le0, z_1\le0,
\\
&(F_2)\quad
z_1+z_2+z_3+3z_4\le0, z_2+3z_4\le0, z_2\le0, z_1> 0,
\\
&(F_3)\quad 
z_1+z_3+3z_4\le0, z_3+3z_4\le0, z_4\le0, z_2> 0, z_1+z_2> 0,
\\
&(F_4)\quad
z_1+z_2+3z_4> 0, z_2+3z_4> 0, z_4> 0, z_3\le0, z_1+z_3\le0,
\\
&(F_5)\quad
z_1+z_2+z_3+3z_4> 0, z_3+3z_4> 0, z_3> 0, z_1\le0, 
\\
&(F_6)\quad
z_1+z_2+z_3+3z_4> 0, z_1+z_3+3z_4> 0, z_1+z_3> 0, z_1> 0.
\end{cases}
\end{equation}
($E_i$) ($1\le i\le 6$) is defined from ($F_i$) by replacing $>$ (resp. $\le$) with $\ge$ (resp. $<$).

Then for $b=(x_1,x_2,x_3,{\bar x}_3,{\bar x}_2,{\bar x}_1) \in B$, we define
$\et_i(b), \ft_i(b), \vep_i(b), \vp_i(b), \\
i= 0, 1, 2$ as follows. We use the convention: $(a)_+=\max(a,0)$.

\begin{align*}
\et_0(b)=&
\begin{cases}
(x_1 -1,\ldots) 
& \text{if ($E_1$)}, 
\\
(\ldots,x_3 -1,{\bar x}_3 -1,\ldots,{\bar x}_1 +1) 
& \text{if ($E_2$)}, 
\\
(\ldots,x_3 -2,\ldots,{\bar x}_2 +1,\ldots) 
& \text{if ($E_3$)}, 
\\ 
(\ldots,x_2 -1,\ldots,{\bar x}_3 +2,\ldots) 
& \text{if ($E_4$)}, 
\\
(x_1 -1,\ldots,x_3 +1,{\bar x}_3 +1,\ldots) 
& \text{if ($E_5$)}, 
\\
(\ldots,{\bar x}_1 +1) 
& \text{if ($E_6$)},
\end{cases}
\\
\ft_0(b)=&
\begin{cases}
(x_1 +1,\ldots) 
& \text{if ($F_1$)}, 
\\ 
(\ldots,x_3 +1,{\bar x}_3 +1,\ldots,{\bar x}_1 -1) 
& \text{if ($F_2$)}, 
\\
(\ldots,x_3 +2,\ldots,{\bar x}_2 -1,\ldots) 
& \text{if ($F_3$)}, 
\\
(\ldots,x_2 +1,\ldots,{\bar x}_3 -2,\ldots) 
& \text{if ($F_4$)}, 
\\
(x_1 +1,\ldots,x_3 -1,{\bar x}_3 -1,\ldots) 
& \text{if ($F_5$)}, 
\\
(\ldots,{\bar x}_1 -1)
& \text{if ($F_6$)},
\end{cases}
\end{align*}

\begin{align*}
\et_1(b)=&
\begin{cases}
(\ldots,{\bar x}_2 +1,{\bar x}_1 -1) 
& \text{if $z_2 \geq (-z_3)_+$}, 
\\  
(\ldots,x_3 +1,{\bar x}_3 -1,\ldots) 
& \text{if $z_2 <0\leq z_3$}, 
\\ 
(x_1 +1,x_2 -1,\ldots) 
& \text{if $(z_2)_+ < -z_3$},
\end{cases}
\\
\ft_1(b)=&
\begin{cases}
(x_1 -1,x_2 +1,\ldots) 
& \text{if $(z_2)_+ \leq -z_3$}, 
\\
(\ldots,x_3 -1,{\bar x}_3 +1,\ldots) 
& \text{if $z_2 \leq 0< z_3$}, 
\\
(\ldots,{\bar x}_2 -1,{\bar x}_1 +1) 
& \text {if $z_2 >(-z_3)_+$},
\end{cases}
\\
\et_2(b)=&
\begin{cases}
(\ldots,{\bar x}_3 + 2,{\bar x}_2 - 1,\ldots) 
& \text{if $z_4 \geq 0$}, 
\\
(\ldots,x_2 + 1,x_3 - 2,\ldots) 
& \text{if $z_4 < 0$},
\end{cases}
\\
\ft_2(b)=&
\begin{cases}
(\ldots,x_2 - 1,x_3 + 2,\ldots) 
& \text{if $z_4 \leq 0$}, 
\\
(\ldots,{\bar x}_3 - 2,{\bar x}_2 + 1,\ldots) 
& \text{if $z_4 > 0$}.
\end{cases}
\end{align*}

\begin{align*}
\vep_1(b) = &\  { \bar x}_1+({\bar x}_3-{\bar x}_2+(x_2-x_3)_+)_+,
\quad \vp_1(b) = 
x_1+(x_3-x_2+({\bar x}_2-{\bar x}_3)_+)_+,
\\
\vep_2(b) = &\ {\bar x}_2+\frac{1}{2}(x_3-{\bar x}_3)_+,
\quad \vp_2(b) = x_2+\frac{1}{2}({\bar x}_3-x_3)_+,\\
\vep_0(b) = & \ 
l-s(b)+\max \,\mathcal{A}-(2z_1+z_2+z_3+3z_4),
\quad \vp_0(b) =
l-s(b)+\max\, \mathcal{A}.
\end{align*}

For $b\in B$ 
if $\et_i(b)$ or $\ft_i(b)$ does not belong to 
$B$, namely, if $x_j$ or $\bar{x}_j$
for some $j$ becomes negative or $s(b)$ exceeds $l$, 
we understand it to be $0$.

The following is one of the main results in 
\cite{KMOY}:
\begin{theorem}\cite{KMOY}
For the quantum affine algebra $U_q(D_4^{(3)})$ the set $B= B^{1,l}$ equipped with the maps $\et_i, \ft_i, \vep_i, \vp_i, i=0, 1, 2$ is a perfect crystal of level $l$.
\end{theorem}
As was shown in \cite{KMOY}, 
the minimal elements are given by:
\[
(B)_{\min}=
\{(\alpha,\beta,\beta,\beta,\beta,\alpha)\,|\,
\alpha, \beta\in \Z_{\geq 0}, 2\alpha+3\beta\leq l\}.
\]


\section{Demazure Modules and Demazure crystals}
\setcounter{equation}{0}
 
Let $W$ denote the Weyl group for the affine Lie algebra $\mathfrak{g}$ generated by the simple reflections $\left\{r_{i} | i \in I \right\}$, where
$r_{i}(\mu)=\mu-\mu(h_{i})\alpha_{i}$ for all $\mu \in \mathfrak{h}^{*}$.  For $w \in W$ let $l(w)$ denote the length
of $w$ and $\prec$ denote the Bruhat order on $W$.  Let $U_q^{+}(\mathfrak{g})$ be the subalgebra of $U_q(\mathfrak{g})$ generated by the $e_{i}$'s.  For $\lambda \in P^{+}$ with $\lambda (d)= 0$,
consider the irreducible integrable highest weight  $U_q(\mathfrak{g})$-module $V(\lambda)$ with highest weight $\lambda$ and highest weight vector $u_{\lambda}$. Let $(L(\lambda), B(\lambda))$ denote the crystal basis for $V(\lambda)$ (\cite{Ka1},\cite{Ka2}).  It is
known that for $w \in W$, the extremal weight space $V(\lambda)_{w\lambda}$ is one dimensional.  Let $V_{w}(\lambda)$ denote the $U_q^{+}(\mathfrak{g})$-module generated by
$V(\lambda)_{w\lambda}$.  These modules $V_w(\lambda)$ $(w \in W)$ are called the {\it Demazure modules}.  They are finite dimensional subspaces of $V(\lambda)$ and satisfy the
properties: $V(\lambda)= \bigcup_{w \in W} V_w(\lambda)$ and for $w,w' \in W$ with $w \preceq w'$ we have $V_w(\lambda) \subseteq V_{w'}(\lambda)$.  

In 1993, Kashiwara \cite{Ka3} showed that for each $w \in W$, there exists a subset $B_{w}(\lambda)$ of $B(\lambda)$ such that
$$ \frac{ V_{w}(\lambda) \cap L(\lambda)}{ V_w(\lambda) \cap qL(\lambda)} = \bigoplus_{b \in B_w(\lambda)}\mathbb{Q}b.$$  
The set $B_w(\lambda)$ is the crystal for the Demazure module $V_w(\lambda)$.  The Demazure crystal
$B_w(\lambda)$ has the following recursive property:
\begin{equation}
\mbox{If } r_iw \succ w, \mbox{ then } B_{r_iw}(\lambda) = \bigcup_{n \ge 0} \tilde{f}_i^nB_w(\lambda) \backslash \left\{ 0 \right\}.
\label{recursive_property}
\end{equation}

Suppose $\lambda(c)=l \ge 1$, and suppose $B$ be a perfect crystal of
level $l$ for $U_q(\mathfrak{g})$.  Then the crystal $B(\lambda)$ is isomorphic to the set of paths $\mathcal{P}(\lambda)= \mathcal{P}(\lambda,B)$ (see Proposition \ref{path_realization}).  Under this isomorphism the highest weight
element $u_{\lambda} \in B(\lambda)$ is identified with the ground state path $\mathbf{p}_\lambda=(\dots \otimes b_3 \otimes b_2 \otimes b_1)$. We recall the path realizations of the Demazure crystals $B_w(\lambda)$ from \cite{KMOU}. Fix
positive integers $d$ and $\kappa$.  For a sequence of integers $\left\{ \left. i_{a}^{(j)} \right| j \ge 1, 1 \le a \le d \right\}
\subset \left\{ 0,1,\dots,n \right\}$, we define the subsets $\left\{ \left. B_a^{(j)} \right| j \ge 1, 0 \le a \le d \right\}$ of $B$ by

$$B_0^{(j)}=\left\{b_j \right\}, B_a^{(j)}= \bigcup_{k \ge 0 } \tilde{f}_{i_a^{(j)}}^kB_{a-1}^{(j)} \backslash \left\{0 \right\}.$$

Next we define $B_a^{(j+1,j)}, j \ge 1, 0 \le a \le d$ by
$$B_0^{(j+1,j)}=B_0^{(j+1)} \otimes B_d^{(j)}, B_a^{(j+1,j)}=\bigcup_{k \ge 0} \tilde{f}_{i_a^{(j+1)}}^k B_{a-1}^{(j+1,j)} \backslash 
\left\{ 0 \right\},$$
and continue until we define
\begin{eqnarray*}
B_0^{(j+ \kappa -1, \dots, j)}&=&B_0^{(j + \kappa -1)} \otimes B_d^{(j + \kappa-2, \dots, j)} \\
B_{a}^{(j + \kappa -1, \dots, j)} &=& \bigcup_{k \ge 0} \tilde{f}_{i_a^{(j+ \kappa-1)}}^kB_{a-1}^{(j + \kappa -1, \dots, j)} \backslash
\left\{0 \right\}. 
\end{eqnarray*}

Furthermore, we define a sequence $\left\{ w ^{(k)} \right\}$ of elements in the Weyl group $W$ by
$$ w^{(0)}=1, w^{(k)}=r_{i_a^{(j)}}w^{(k-1)}, \mbox{ for }k >0.$$
Here, $j$ and $a$ are uniquely determined from $k$ by the relation $k=(j-1)d + a, j \ge 1, 1 \le a \le d$.  Now for $k \ge 0$,
we define subsets $\mathcal{P}^{(k)}(\lambda,B)$ of $\mathcal{P}(\lambda,B)$ as follows:
\begin{eqnarray*}
\mathcal{P}^{(0)}(\lambda,B) &=& \left\{ \mathbf{p}_\lambda \right\}, \\
\mbox{ and for } k > 0, \\
\mathcal{P}^{(k)}(\lambda,B) &=& \begin{cases} \dots \otimes B_0^{(j+2)} \otimes B_0^{(j+1)} \otimes B_a^{(j, \dots, 1)} & \mbox{if }j < \kappa \\
\dots \otimes B_{0}^{(j+2)} \otimes B_0^{(j+1)} \otimes B_a^{(j, \dots, j-\kappa +1)} \otimes B^{\otimes (j-\kappa)} 
& \mbox{if } j \geq \kappa \end{cases}.  
\end{eqnarray*}

The following theorem shows that under certain conditions, the path realizations of the Demazure crystals $B_{w^{(k)}}(\lambda)$ is isomorphic to $\mathcal{P}^{(k)}(\lambda,B)$ and hence have tensor product-like structures.

\begin{theorem}\label{tensor} \cite{KMOU} Let $\lambda \in {\bar P}^{+}$ with $\lambda (c)= l$ and $B$ be a perfect crystal of level $l$ for the quantum affine
Lie algebra $U_q(\mathfrak{g})$.  For fixed positive integers $d$ and $\kappa$, suppose we have a sequence of integers
$\left\{ i_a^{(j)} | j \ge 1, 1 \le a \le d \right\} \subset \left\{ 0,1,2, \dots, n \right\}$ satisfying the conditions:
\begin{enumerate}
\item for any $j \ge 1 $, $B_d^{(j + \kappa -1, \dots, j)} = B_d^{(j + \kappa -1, 
\dots, j+1)}	 \otimes B$,
\item for any $j \ge 1, 1 \le a \le d$, $\left< \lambda_j, h_{i_a^{(j)}} \right> \le 
\varepsilon_{i_a^{(j)}}(b), b \in B_{a-1}^{(j)}$, and 
\item the sequence of elements $\left\{ w^{(k)} \right\}_{k \ge 0}$ is an increasing sequence with respect to the Bruhat order.
\end{enumerate}
Then we have
$B_{w^{(k)}}(\lambda) \cong \mathcal{P}^{(k)}(\lambda,B).$
\label{affine_demazure_theorem}
\end{theorem}

The positive integer $\kappa$ in Theorem \ref{affine_demazure_theorem} is called the {\it mixing index}. It is conjectured that for any affine Lie algebra $\mathfrak{g}$, the mixing index $\kappa \le 2$. It is known that the mixing index $\kappa$ is dependent on the choice of the perfect crystal (see \cite{MW}). For $ \lambda = l\Lambda$, (where $\Lambda$ is a dominant weight of level one) and the perfect crystal $B = B^{1,l}$ it is known that there exists a suitable sequence of Weyl group elements $\left\{w^{(k)}\right\}$ which satisfy the conditions in Theorem \ref{affine_demazure_theorem} with $\kappa = 1$ for $\mathfrak{g}$ any classical quantum affine Lie algebra \cite{KMOTU}, $U_q(D_4^{(3)})$ \cite{M}, and $U_q(G_2^{(1)})$ \cite{JM}.

For $b \in B$, let $\tilde{f}_i^{\rm max}(b)$ denote $\tilde{f}_i^{\varphi_i(b)}(b)$. For $j \ge 1$, we set 
$$b_0^{(j)} =b_j, \quad b_a^{(j)} = \tilde{f}_{i_a^{(j)}}^{\rm max}(b_{a-1}^{(j)}) \quad (a = 1, 2, \cdots, d).$$

The following Proposition (\cite{KMOU}, Proposition 2) will be useful to check the validity of condition $(3)$ in Theorem \ref{affine_demazure_theorem}.

\begin{proposition} \label{bruhat} \cite{KMOU} For $w \in W$, if $\langle w\mu, h_j \rangle > 0$ for some $\mu \in {\bar P}^+$, then $r_jw \succ w$.
\end{proposition}

\section{$U_q(D_4^{(3)})$-Demazure crystals}
\setcounter{equation}{0}

In this section we show that for the perfect crystal of level $3l$ for the quantum affine algebra $U_q(D_4^{(3)})$ given in Section 3, there is a   suitable sequence of Weyl group elements $\left\{ w^{(k)} \right\}$ which satisfy the conditions  $(1), (2)$ and $(3)$ for $\lambda= l\Lambda_2$ and hence Theorem \ref{affine_demazure_theorem} holds in this case with $\kappa =1$. Thus we have path realizations of the corresponding Demazure crystals with tensor product-like structures.

For $\lambda = l \Lambda_2, l \ge 1$, the $l\Lambda_2$-minimal element in the perfect crystal $B= B^{1,3l}$ is ${\bar b} = (0, l, l, l, l, 0)$  and  in this case $\lambda_j = \lambda = l\Lambda_2$ for $j \ge 1$. Hence $b_j = {\bar b}$ for all $j \ge 1$. Set $d =6$ and choose the sequence $\{ i_a^{(j)} | j \ge 1, 1 \le a \le 6\}$ defined by 

\begin{equation}\label{sequence}
i_1^{(j)}= i_3^{(j)} = 2, i_2^{(j)}= i_4^{(j)}= i_6^{(j)}= 1,  i_5^{(j)}= 0.
\end{equation}
Hence, by the action of $\tilde{f_i}$ on $B$ we have, for $j \ge 1$
\begin{eqnarray}
\begin{array}{l}
b_0^{(j)}= (0, l, l, l, l, 0), \\
b_1^{(j)}= \tilde{f}_2^{\rm max}(b_0^{(j)})= (0, 0, 3l, l, l, 0),\\
b_2^{(j)}= \tilde{f}_1^{\rm max}(b_1^{(j)})= (0, 0, 0, 4l, l, 0),\\
b_3^{(j)}= \tilde{f}_2^{\rm max}(b_2^{(j)})= (0, 0, 0, 0, 3l, 0),\\
b_4^{(j)}= \tilde{f}_1^{\rm max}(b_3^{(j)})= (0, 0, 0, 0, 0, 3l),\\
b_5^{(j)}= \tilde{f}_0^{\rm max}(b_4^{(j)})= (3l, 0, 0, 0, 0, 0),\\
b_6^{(j)}= \tilde{f}_1^{\rm max}(b_5^{(j)})= (0, 3l, 0, 0, 0, 0),
\end{array}
\label{equation1}
\end{eqnarray}

We define conditions $P$ and $Q_j, 1 \leq j \leq 6$  for $b \in B$ as follows:

\begin{eqnarray*}
\begin{cases}
(P): z_3 \geq 0; z_3 + 3z_4 \geq (-2z_2)_+; z_1 + z_2 + z_3 +3z_4 \geq 0; t(b) < 2l; s(b) < 3l\\
(Q_1): z_3 < 0; z_4 \geq 0; z_1 +z_2 +3z_4 \geq 0; z_1 + 2z_2 + z_3 +3z_4 \geq 0; t(b) \leq 2l; s(b) \leq 3l\\
(Q_2): z_2 \geq 0; z_4 < 0; z_3 + 3z_4 < 0; z_1 + z_2 \geq 0; z_1 + 2z_2 +z_3  \geq 0; s(b) \leq 3l\\
(Q_3): z_2 \geq 0; z_4 < 0; z_3 + 3z_4 < 0; z_1 + z_2 < 0; z_2 + z_3 \geq 0; s(b) \leq 3l\\
(Q_4): z_2 \geq 0; z_3 \geq 0; z_1 + z_2 + z_3 + 3z_4 < 0; z_3 + 3z_4 \geq 0; t(b) < 2l; s(b) \leq 3l\\
(Q_5): x_1 > 0; z_3 \geq 0; z_3 + 3z_4 \geq 0; z_1 + z_2 + z_3 + 3z_4 \geq 0; z_1 + 2z_2 + z_3 + 3z_4 \geq 0;\\
\qquad \ \ t(b) < 2l; s(b) \leq 3l\\ 
(Q_6): x_1 > 0; z_3 < 0; z_4 \geq 0; z_1 + z_2 + 3z_4 < 0; z_2 + 3z_4 > 0; z_2 + z_3 \geq 0; t(b) < 2l; s(b) \leq 3l
\end{cases}
\end{eqnarray*}

By direct calculations it can be seen that the subsets $\left\{ \left. B_a^{(j)} \right| j \ge 1, 0 \le a \le 6 \right\}$ of $B$ are given as follows.

\begin{eqnarray*}
\begin{array}{l}
B_0^{(j)} = \{(0, l, l, l, l, 0)\} \\
B_1^{(j)} = B_0^{(j)} \cup \{(0, x_2, x_3, l, l, 0) \mid  z_3 > 0, s(b) = 3l\} \\
B_2^{(j)} = B_1^{(j)} \cup \{(0, x_2, x_3, {\bar x}_3, l, 0) \mid  z_2 < 0, z_3 \geq 0, s(b) = 3l\} \\
B_3^{(j)} = B_2^{(j)} \cup \{(0, x_2, x_3, {\bar x}_3, {\bar x}_2, 0) \mid z_3 \geq 0, z_3 + 3z_4 \geq 0, t(b) < 2l, s(b) = 3l\} \\
B_4^{(j)} = B_3^{(j)} \cup \{(0, x_2, x_3, {\bar x}_3, {\bar x}_2, {\bar x}_1) \mid {\bar x}_1 > 0, z_3 \geq 0, z_3 + 3z_4 \geq (-2z_2)_+, t(b) < 2l, s(b) = 3l\} \\
B_5^{(j)} = B_4^{(j)} \cup C \cup D_1 \cup D_2 \cup \ldots D_6\\
B_6^{(j)} = B,
\end{array}
\end{eqnarray*}
where $C = \{(0, x_2, x_3, {\bar x}_3, {\bar x}_2, {\bar x}_1) \mid (P) \ \  {\rm holds}\}$ and for $1 \leq j \leq 6$, \\ $D_j = \{(x_1, x_2, x_3, {\bar x}_3, {\bar x}_2, {\bar x}_1) \mid (Q_j) \ \  {\rm holds}\}$.

The following lemma is useful and follows by easy calculations.

\begin{lemma}\label{weyl}
Let $k \in \Z _{>0}$ and $k = 6(j-1) + a, j\geq 1, 1 \leq a \leq 6.$  Then we have $w^{(k)}\Lambda_{2} = \Lambda_{2} - m_{0}\alpha_{0} - m_{1}\alpha_{1} - m_{2}\alpha_{2}$ where\\
$m_{0} = 
\begin{cases}
3j^{2} + 3j & \text{if } a = 1, 2, 3, 4 \\
3j^2 + 9j +6  & \text{if } a = 5, 6
\end{cases}$\\
$m_{1} = 
\begin{cases}
6j^{2} + 3j,     & \text{if } a = 1 \\
6j^{2} + 9j + 3,     & \text{if } a = 2, 3 \\
6j^{2}  + 12j + 6,     & \text{if } a = 4,5 \\
3j^{2} + 15j + 9,     & \text{if } a = 6 \\
\end{cases}$\\
$m_{2} = 
\begin{cases}
3j^{2} + 3j +1,     & \text{if } a = 1,2 \\
3j^{2} + 6j +3,     & \text{if } a = 3, 4, 5,6 \\
\end{cases}$
\end{lemma}

\begin{theorem} For $\lambda = l\Lambda_2$, $l \ge 1$ and the given perfect crystal $B = B^{1, 3l}$ for the quantum affine algebra $U_q(D_4^{(3)})$ with $d = 6$ and the  sequence $\{i_a^{(j)}\}$ given in (\ref{sequence}), conditions (1), (2) and (3) in Theorem \ref{tensor} hold with $\kappa = 1$. Hence we have path realizations of the corresponding Demazure crystals $B_{w^{(k)}}(l\Lambda_2)$ for  $U_q(D_4^{(3)})$ 
with tensor product-like structures.
\end{theorem}
\begin{proof}
We have already shown above by explicit descriptions of the subsets $B_a^{(j)}$ that $B_6^{(j)} = B$ which implies that condition $(1)$  in Theorem \ref{tensor} holds for $\kappa = 1$.

Observe that  $\left<l\Lambda_2 , h_{i_a^{(j)}}\right> = 0 \le \varepsilon_{i_a^{(j)}}(b)$ for all $b \in B_{a-1}^{(j)}, a = 2, 4, 5, 6$.  Also $\left<l\Lambda_2 , h_{i_a^{(j)}}\right> = l$ for $a = 1, 3$. 
Observe that for all $b = (x_1,x_2,x_3,{\bar x}_3,{\bar x}_2,{\bar x}_1) \in B_0^{(j)} \ \ {\rm or} \ \ B_2^{(j)}$ we have ${\bar x}_2 = l$. Hence $\varepsilon_{i_a^{(j)}}(b) = {\bar x}_2 + \frac{1}{2}(x_3 - {\bar x}_3)_+ = l + \frac{1}{2}(x_3 - {\bar x}_3)_+ \geq l$ for $a = 1, 3$ and condition $(2)$ holds.

To prove condition (3), we use Lemma \ref{weyl} to obtain:\\
$\left <w^{(k)}\Lambda_{2}, h_{i_{a+1}^{(j)}} \right > \ = 
\begin{cases}
3j + 2    & \text{if } a = 2 \\
3j + 3     & \text{if } a = 3, 5\\
3j + 4 & \text{if } a = 6\\
6j + 3     & \text{if } a = 1,\\
6j + 6     & \text{if } a = 4,\\
\end{cases}$\\
where $k= 6(j-1) + a, j \geq 1$. Hence $\left <w^{(k)}\Lambda_{2}, h_{i_{a+1}^{(j)}} \right > >0 $ for all $ j \geq 1$. Therefore, by Proposition \ref{bruhat}, $w^{(k+1)}=r_{i_{a+1}^{(j)}}w^{(k)} \succ w^{(k)}$, which implies that condition $(3)$ holds.

\end{proof}

\bibliographystyle{amsalpha}


\end{document}